\documentclass{amsart}
\usepackage{mathrsfs}
\usepackage{amsfonts}
\usepackage{amsfonts,amssymb,amsmath,amsthm}
\usepackage{url}
\usepackage{enumerate}

\newtheorem{thrm}{Theorem}[section]
\newtheorem{lem}[thrm]{Lemma}
\newtheorem{prop}[thrm]{Proposition}
\newtheorem{cor}[thrm]{Corollary}
\theoremstyle{definition}

\newtheorem{remark}[thrm]{Remark}

\newenvironment{proof'}{{\em Proof of Main Theorem.} }{\hfill$\Box$\vspace{0.05in}}

\numberwithin{equation}{section}

\title[MEAN CURVATURE FLOW OF HIGHER CODIMENSION]{\large{Mean Curvature Flow of Higher Codimension in Hyperbolic Spaces }}

\author{Kefeng Liu}
\address{Center of Mathematical Sciences, Zhejiang University,
             Hangzhou,  310027, People¡¯s Republic of China; Department of Mathematics, UCLA, Box 951555, Los Angeles, CA, 90095-1555 }
\email{liu@cms.zju.edu.cn, liu@math.ucla.edu}

\author{Hongwei Xu}

\address{Center of Mathematical Sciences, Zhejiang University,
             Hangzhou,  310027, People¡¯s Republic of China}
\email{xuhw@cms.zju.edu.cn}

\author{Fei Ye}
\address{Center of Mathematical Sciences, Zhejiang University,
             Hangzhou,  310027, People¡¯s Republic of China}
\email{yf@cms.zju.edu.cn}

\author{Entao Zhao}
\address{Center of Mathematical Sciences, Zhejiang University,
             Hangzhou,  310027, People¡¯s Republic of China}
\email{zhaoet@cms.zju.edu.cn}

\thanks{Research supported by the National Natural Science Foundation
of China, Grant No. 11071211; the Trans-Century Training Programme
Foundation for Talents by the Ministry of Education of China, and
the China Postdoctoral Science Foundation, Grant No. 20090461379.}


\begin{document}

\begin{abstract} In this paper we investigate the convergence for the mean curvature flow of closed submanifolds
with arbitrary codimension in space forms.  Particularly, we prove
that the mean curvature flow deforms a closed submanifold satisfying
a pinching condition in a hyperbolic space form to a round point in
finite time.
\end{abstract}
\maketitle

\section{Introduction}
In this paper, we study the convergence of the mean curvature flow
of submanifolds in space forms. Let $F:M^{n}\rightarrow
\mathbb{F}^{n+d}(c)$ be a smooth immersion from an $n$-dimensional
closed Riemannian manifold $M^{n}$ to an $(n+d)$-dimensional
complete simply connected space form $\mathbb{F}^{n+d}(c)$ with
constant sectional curvature $c$. Consider a one-parameter family of
smooth immersions $F:M\times [0,T)\rightarrow \mathbb{F}^{n+d}(c)$
satisfying
\begin{eqnarray}
\label{MCF}\left\{
\begin{array}{ll}
\frac{\partial}{\partial t}F(x,t)&=\ H(x,t),\\
F(x,0)&=\ F(x),
\end{array}\right.
\end{eqnarray}
where  $H(x,t)$ is the mean curvature vector of $F_t(M)$ and
$F_t(x)=F(x,t)$. We call $F:M\times [0,T)\rightarrow
\mathbb{F}^{n+d}(c)$ the mean curvature flow with initial value $F$.

The mean curvature flow was proposed  by Mullins \cite{Mullins56} to
describe the formation of grain boundaries in annealing metals. In
\cite{B}, Brakke introduced the motion of a submanifold by its mean
curvature in arbitrary codimension and constructed a generalized
varifold solution for all time. For the classical solution of the
mean curvature flow, most works have been done on hypersurfaces.
Huisken \cite{H1,H2} showed that if the initial hypersurface in a
Riemannian manifold is uniformly convex,  then the mean curvature
flow converges to a round point in finite time. Later, Huisken
\cite{H3} extend this result to hypersurfaces satisfying a pinching
condition in a sphere. Many other beautiful results have been
obtained, and there are various approaches to study the mean
curvature flow of hypersurfaces (see \cite{CGG,ES}, etc.). For the
mean curvature flow of submanifolds in higher codimension, some
special cases have been studied,  see
\cite{Sm,Sm2,SW,WaM1,WaM2,WaM3} etc. for example. Recently,
Andrews-Baker \cite{Andrews-Baker} proved a convergence theorem for
the mean curvature flow of closed submanifolds satisfying a pinching
condition in the Euclidean space. In \cite{Baker}, Baker proved a
convergence result for the mean curvature flow of submanifolds in a
sphere. In this paper, we study the mean curvature flow of closed
submanifolds in hyperbolic spaces and extend the convergence result
in \cite{Andrews-Baker,Baker} to the mean curvature flow of
arbitrary codimension in space forms.
\begin{thrm}\label{convergence-c}
Let $F:M^{n}\rightarrow \mathbb{F}^{n+d}(c)$ be a smooth closed
submanifold in a hyperbolic space with constant curvature $c<0$.
Assume $F$ satisfies
\begin{eqnarray}
\label{pinch-cond}
|A|^2\leq\begin{cases}
            \frac{4}{3n}|H|^2+\frac{n}{2}c, \ &n = 2, 3, \\
            \frac{1}{n-1}|H|^2+2c, \ &n \geq 4.
        \end{cases}
\end{eqnarray}
Then the mean curvature flow with $F$ as initial value converges to
a round point in finite time.
\end{thrm}

As an immediate consequence of Theorem \ref{convergence-c}, we
obtain the following differentiable sphere theorem.

\begin{cor}\label{coro-sphere thm}
Let $F:M^{n}\rightarrow \mathbb{F}^{n+d}(c)$ be a smooth closed
submanifold in a hyperbolic space with constant curvature $c<0$.
Assume $F$ satisfies
\begin{eqnarray*}
|A|^2\leq\begin{cases}
            \frac{4}{3n}|H|^2+\frac{n}{2}c, \ &n = 2, 3, \\
            \frac{1}{n-1}|H|^2+2c, \ &n \geq 4.
        \end{cases}
\end{eqnarray*}
Then $M$ is diffeomorphic to the unit $n$-sphere.
\end{cor}

\begin{remark}This differentiable sphere theorem was also obtained by Gu and Xu \cite{Gu-Xu,Xu-Gu}
provided the submanifold is simply connected. In fact, they proved
the sphere theorem for submanifolds in a Riemannian manifold  by
using a different method. For more sphere theorems of submanifolds,
we refer the readers to
\cite{Andrews-Baker,Baker,Fu-Xu,Gu-Xu,LXYZ,Shiohama-Xu-97,Xu-Gu,Xu-Zhao},
etc.
\end{remark}

Combining Theorem \ref{convergence-c} and the convergence results in
\cite{Andrews-Baker,Baker}, we obtain the following theorem.
\begin{thrm}\label{convergence-spaceform}Let $F:M^{n}\rightarrow
\mathbb{F}^{n+d}(c)$ be a smooth closed submanifold in a complete
simply connected space form with $|H|^2+n^2c>0$. Assume $F$
satisfies
\begin{eqnarray}
\label{pinch-cond-1}|A|^2\leq\begin{cases}
            \frac{4}{3n}|H|^2+\frac{1}{12}[7n-4+{\rm sgn}(c)(n-4)]c, \ &n = 2, 3, \\
            \frac{1}{n-1}|H|^2+2c, \ &n \geq 4.
        \end{cases}
\end{eqnarray}
Then either $F_t(M)$ converges to a round point in finite time, or
$c>0$ and $F_t(M)$ converges to a total geodesic sphere in
$\mathbb{F}^{n+d}(c)$ as $t\rightarrow \infty$.
\end{thrm}

\begin{remark}For $c>0$, $|H|^2+n^2c>0$ is automatically
satisfied. For $c=0$, $|H|^2+n^2c>0$  is equivalent to that the mean
curvature is nowhere vanishing. For $c<0$, $|H|^2+n^2c>0$  is
implied by condition (\ref{pinch-cond-1}).
\end{remark}

\begin{remark}For $c>0$, the maximal existence time of the mean curvature flow may be finite or infinite. For $c\leq0$, the mean curvature flow with a closed
initial submanifold always has finite maximal existence time.
\end{remark}

\section{Basic equations}
Let $F:M\times [0,T)\rightarrow \mathbb{F}^{n+d}(c)$ be a smooth
mean curvature flow with initial closed immersion $F_0:M\rightarrow
\mathbb{F}^{n+d}(c)$. Denote by $g(t)$ and  $d\mu_t$ the induced
metric and the volume form on $M$. Let $A$ and $H$ be the second
fundamental form and the mean curvature vector of $M$ in
$\mathbb{F}^{n+d}(c)$, respectively. We shall make use of the
following convention on the range of indices.
$$1\leq i,j,k,\cdots \leq n,\ \ 1\leq A,B,C,\cdots \leq n+d\ \
and\ \ n+1\leq\alpha,\beta,\gamma, \cdots \leq n+d.$$

As in \cite{Andrews-Baker,B}, we consider the evolution on the
spatial tangent bundle. Choose a local orthonormal frame $\{e_i\}$
for the spatial tangent bundle and a local orthonormal frame
$\{\nu_\alpha\}$ for the normal bundle. Let $\{\omega_i\}$ be the
dual frame of  $\{e_i\}$. Then $A$ and $H$ can be written as
\begin{eqnarray*}
A=\sum_{i,j,\alpha}h_{ij\alpha}\omega_i\otimes\omega_j\otimes\nu_\alpha=\sum_{i,j}h_{ij}\omega_i\otimes\omega_j,
\ \ H=\sum_\alpha H_{\alpha}\nu_\alpha.\end{eqnarray*}

We have the following evolution equations.
\begin{eqnarray}\label{|A|}
    \frac{\partial}{\partial t}|A|^2 = \Delta|A|^2 - 2|\nabla A|^2 + 2R_1 + 4c|H|^2 -
    2nc|A|^2,\end{eqnarray}
\begin{eqnarray}\label{|H|}
   \frac{\partial}{\partial t}|H|^2 =\Delta|H|^2 - 2|\nabla H|^2 + 2R_2 +
   2nc|H|^2,
\end{eqnarray}
where
\begin{eqnarray}\label{R_1}
R_1=\sum_{\alpha,\beta}\Big(\sum_{i,j}h_{ij\alpha}h_{ij\beta}\Big)^2+|R^\bot|^2,
\end{eqnarray}
\begin{eqnarray}\label{R_bar}
|R^\bot|^2=\sum_{i,j,\alpha,\beta}\Big(\sum_p\Big(h_{ip\alpha}h_{jp\beta}-h_{jp\alpha}h_{ip\beta}\Big)\Big)^2,
\end{eqnarray}
\begin{eqnarray}\label{R_2}
R_2=\sum_{i,j}\Big(\sum_{\alpha} H_{\alpha}h_{ij\alpha}\Big)^2.
\end{eqnarray}
The contracted form of Simons' identity for traceless second
fundamental form $\mathring{A}:=A-\frac{1}{n} g\otimes H$ is
\begin{eqnarray}\label{Simons identity}
\frac{1}{2}\triangle
|\mathring{A}|^2=\mathring{h}_{ij}\nabla_i\nabla_jH+|\nabla
\mathring{A}|^2+Z+nc|\mathring{A}|^2.
\end{eqnarray}
Here
\begin{eqnarray}\label{Z}
Z=-R_1+\sum_{i,j,p,\alpha,\beta}H_\alpha
h_{ip\alpha}h_{ij\beta}h_{pj\beta}.
\end{eqnarray}

We also have the following inequality.
\begin{eqnarray}\label{nabla A}
|\nabla A|^2\geq\frac{3}{n+2}|\nabla H|^2.
\end{eqnarray}

\section{Preserved curvature pinching condition}
Now we prove that the pinching condition (\ref{pinch-cond}) is
preserved under the mean curvature flow with arbitrary codimension
in the hyperbolic space.
\begin{lem}\label{preserved pinch}For $c<0$ and $n\geq2$, if the initial immersion satisfies
(\ref{pinch-cond}), then this condition is preserved along the mean
curvature flow.
\end{lem}
\begin{proof}We consider $Q=|A|^2-\alpha|H|^2-\beta c$, where the constants
\begin{equation*}
        \alpha\leq \begin{cases}
            \frac{4}{3n}, \ &n = 2, 3, \\
            \frac{1}{n-1}, \ &n \geq 4,
        \end{cases}\ \text{and } \
        \beta\geq
            \begin{cases}
                 \frac{n}{2}, \ &n = 2, 3, \\
                2, \ &n \geq 4.
            \end{cases}
\end{equation*}

By (\ref{|A|}) and (\ref{|H|}) we have
\begin{eqnarray}\label{Q}
\frac{\partial}{\partial t}Q&=&\triangle Q-2(|\nabla
A|^2-\alpha|\nabla H|^2)\nonumber\\
&&+2R_1-2\alpha
R_2-2nc|\mathring{A}|^2-2n\bigg(\alpha-\frac{1}{n}\bigg)c|H|^2.
\end{eqnarray}

We only have to show that if $Q=0$ at a point $x\in M$, then
\begin{eqnarray*}2R_1-2\alpha
R_2-2nc|\mathring{A}|^2-2n\bigg(\alpha-\frac{1}{n}\bigg)c|H|^2\leq0\end{eqnarray*}
holds at $x$. We also have $H\neq0$ at $x$. Choose $\{\nu_\alpha\}$
such that $\nu_{n+1}=\frac{H}{|H|}$. Let
$A_H=\sum_{i,j}h_{ij,n+1}\omega_i\otimes\omega_j$. Set
$\mathring{A}_H=A_H-\frac{|H|}{n}{\rm Id}$  and
$|\mathring{A}_I|^2=|\mathring{A}|^2-|\mathring{A}_H|^2$.

We replace $|H|^2$ with $\frac{|\mathring{A}|^2-\beta
c}{\alpha-\frac{1}{n}}$. Then
\begin{equation}\label{1}
\begin{split}
&2R_1-2\alpha
R_2-2nc|\mathring{A}|^2-2n\bigg(\alpha-\frac{1}{n}\bigg)c|H|^2\\
\leq&2|\mathring{A}_H|^2-2\bigg(\alpha-\frac{1}{n}\bigg)|\mathring{A}_H|^2|H|^2
+\frac{2}{n}|\mathring{A}_H|^2|H|^2-\frac{2}{n}\bigg(\alpha-\frac{1}{n}\bigg)|H|^4\\
&+8|\mathring{A}_H|^2|\mathring{A}_I|^2+3|\mathring{A}_I|^2-2nc(|\mathring{A}_H|^2+|\mathring{A}_I|^2)
-2n\bigg(\alpha-\frac{1}{n}\bigg)c|H|^2\\
=&\bigg(6-\frac{2}{n(\alpha-\frac{1}{n})}\bigg)|\mathring{A}_H|^2|\mathring{A}_I|^2
+\bigg(3-\frac{2}{n(\alpha-\frac{1}{n})}\bigg)|\mathring{A}_I|^4\\
&+\bigg(2\beta-4n+\frac{2\beta}{n(\alpha-\frac{1}{n})}\bigg)c|\mathring{A}_H|^2
+4\bigg(\frac{\beta}{n(\alpha-\frac{1}{n})}-n\bigg)c|\mathring{A}_I|^2\\
&-2\beta\bigg(\frac{\beta}{n(\alpha-\frac{1}{n})}-n\bigg)c^2.
\end{split}
\end{equation}
By the definition of $\alpha$ and $\beta$, we know that the right
hand side of (\ref{1}) is nonpositive for $n\geq2$. This completes
the proof of the lemma.
\end{proof}

For $\epsilon>0$, set

\begin{equation*}
        \alpha_\epsilon= \begin{cases}
            \frac{4}{3n+n\epsilon}, \ &n = 2, 3, \\
            \frac{1}{n-1+\epsilon}, \ &n \geq 4,
        \end{cases}\ \text{and } \
        \beta_\epsilon=
            \begin{cases}
                 \frac{n}{2}(1+\epsilon), \ &n = 2, 3, \\
                2(1+\epsilon), \ &n \geq 4.
            \end{cases}
\end{equation*}
If the initial immersion satisfies
$|A|^2<\frac{4}{3n}|H|^2+\frac{n}{2}c$ for $n=2,3$, and
$|A|^2<\frac{1}{n-1}|H|^2+2 c$ for $n\geq4$, then there exists an
$\epsilon>0$ such that $|A|^2\leq\alpha_\epsilon|H|^2+\beta_\epsilon
c$ holds on $M_0$. From the proof of Lemma \ref{preserved pinch},
this inequality also holds for $t>0$. On the other hand, if
$|A|^2=\frac{4}{3n}|H|^2+\frac{n}{2}c$ for $n=2,3$, or
$|A|^2=\frac{1}{n-1}|H|^2+2 c$ for $n\geq4$ holds somewhere on
$M_0$, then by the maximum principle, we see that either the
equality holds everywhere on $M_0$, or the strict inequality holds
everywhere for $t>0$. For the first case, we have $\nabla A=0$ and
$\mathring{A}_I=0$ on $M_0$. By \cite{Erbacher}, $M_0$ lies in an
$(n+1)$-dimensional total geodesic submanifold of
$\mathbb{F}^{n+d}(c)$. Since $\nabla A=0$, from Theorem 4 of
\cite{L}, $M_0$ is either locally isometric to an Euclidean space,
or locally isometric to a  product $\mathbb{F}^k(c_1)\times
\mathbb{F}^{n-k}(c_2)$ for some $c_1>0$, $c_2<0$ and $k=0,\cdots,
n$. Since $M_0$ is closed, we see that $M_0$ is a totally umbilical
sphere. Then $|A|^2\leq\alpha_\epsilon|H|^2+\beta_\epsilon c$ holds
on $M_0$ for some $\epsilon>0$. For the second case, we see that
after a short time, we also have
$|A|^2\leq\alpha_\epsilon|H|^2+\beta_\epsilon c$ for some
$\epsilon>0$. Hence, we may assume that
$|A|^2\leq\alpha_\epsilon|H|^2+\beta_\epsilon c$ for some
$\epsilon\in(0,1)$ and $t\geq t_0>0$.

\section{Pinching of $\mathring{A}$ along the mean curvature flow}

Assume that $c<0$. We prove a pinching estimate  for the traceless
second fundamental form, which guarantees that $M_t$ becomes
spherical along the mean curvature flow.

\begin{thrm}\label{thm1}There are positive constants $C_0$ and $\sigma_0$ independent of $t$ such
that\begin{eqnarray}
\label{mathring{A}-pinching}|\mathring{A}|^2\leq C_0|H|^{2-\sigma_0}
\end{eqnarray}
holds along the mean curvature flow.
\end{thrm}
\begin{proof}
We consider the function
$f_\sigma=\frac{|\mathring{A}|^2}{(a|H|^2+\beta_\epsilon
c)^{1-\sigma}}$, where $\sigma\in (0,1)$ and
\begin{equation*}
        a= \begin{cases}
            \frac{1}{3n+n\epsilon}, \  &n = 2, 3 \\
            \frac{1}{n(n-1+\epsilon)}, \  &n \geq 4.
        \end{cases}
\end{equation*}
Notice that
\begin{equation}\label{>0-1}
\begin{split}
&\ a|H|^2+\beta_\epsilon
c-\bigg(\Big(\alpha_\epsilon-\frac{1}{n}\Big)|H|^2+\beta_\epsilon
c\bigg)\\
\geq&\
\frac{\epsilon}{3n+n\epsilon}|H|^2\\
\geq&\
\frac{\epsilon}{3n+n\epsilon}\cdot\beta(-c)\\
>&\ 0 \end{split}
\end{equation} for $n=2,3$, and
\begin{equation}\label{>0-2}
\begin{split}
&\ a|H|^2+\beta_\epsilon
c-\bigg(\Big(\alpha_\epsilon-\frac{1}{n}\Big)|H|^2+\beta_\epsilon
c\bigg)\\
\geq&\
\frac{\epsilon}{n(n-1+\epsilon)}|H|^2\\
\geq&\
\frac{\epsilon}{n(n-1+\epsilon)}\cdot\beta(-c)\\
>&\ 0
\end{split}
\end{equation} for $n\geq4$. So $f_\sigma$ is well-defined. From
(\ref{>0-1}) and (\ref{>0-2}) we also have
\begin{eqnarray}\label{aH2}
a|H|^2+\beta_\epsilon c\geq b|H|^2,
\end{eqnarray}
where
\begin{equation*}
        b= \begin{cases}
            \frac{\epsilon}{3n+n\epsilon}, \  &n = 2, 3 \\
            \frac{\epsilon}{n(n-1+\epsilon)}, \  &n \geq 4.
        \end{cases}
\end{equation*}

By a similar computation as in \cite{Andrews-Baker}, we have
\begin{equation}
\begin{split}
\label{f-sigma} \frac{\partial}{\partial t}f_{\sigma}=&\triangle
f_{\sigma}+\frac{2a(1-\sigma)}{a|H|^2+\beta_\epsilon c}\langle
\nabla |H|^2,\nabla f_{\sigma}\rangle\\
&-\frac{2}{(a|H|^2+\beta_\epsilon c)^{1-\sigma}}\bigg(|\nabla
A|^2-\frac{1}{n}|\nabla
H|^2-\frac{a|\mathring{A}|^2}{a|H|^2+\beta_\epsilon
c}|\nabla H|^2\bigg)\\
&-\frac{4a^2\sigma(1-\sigma)}{(a|H|^2+\beta_\epsilon
c)^2}f_\sigma|H|^2\cdot \Big|\nabla |H|\Big|^2-\frac{2a\sigma
f_\sigma}{a|H|^2+\beta_\epsilon c}|\nabla H|^2\\
&+\frac{2}{(a|H|^2+\beta_\epsilon
c)^{1-\sigma}}\bigg(R_1-\frac{1}{n}R_2-\frac{aR_2|\mathring{A}|^2}{a|H|^2+\beta_\epsilon
c}-nc|\mathring{A}|^2\\
&-\frac{an(1-\sigma)c|\mathring{A}|^2|H|^2}{a|H|^2+\beta_\epsilon
c}\bigg)\\
&+\frac{2a\sigma R_2f_\sigma}{a|H|^2+\beta_\epsilon c}.
\end{split}
\end{equation}

By (\ref{nabla A}), we have
\begin{equation}\label{2}
\begin{split}
&\ |\nabla A|^2-\frac{1}{n}|\nabla
H|^2-\frac{a|\mathring{A}|^2}{a|H|^2+\beta_\epsilon c}|\nabla
H|^2\\
\geq&\ \bigg(\frac{3}{n+2}-\frac{1}{n} -
\frac{a\Big((\alpha_\epsilon-\frac{1}{n})H^2+\beta_\epsilon
c\Big)}{a|H|^2+\beta_\epsilon c} \bigg)|\nabla H|^2\\
\geq&\ \bigg(\frac{3}{n+2}-\frac{1}{n} -a \bigg)|\nabla
H|^2\\
:=&\ \epsilon_\nabla|\nabla H|^2.
\end{split}
\end{equation}
Here $\epsilon_\nabla$ is a positive constant for $n\geq 2$.

We also have the following estimate.
\begin{equation}\label{3}
\begin{split}
&\
R_1-\frac{1}{n}R_2-\frac{aR_2|\mathring{A}|^2}{a|H|^2+\beta_\epsilon
c}\\
\leq&\
R_1-\frac{1}{n}R_2-\frac{R_2|\mathring{A}|^2}{|H|^2}\\
=&\ R_1-\frac{R_2|A|^2}{|H|^2}\\
\leq&\
|\mathring{A}_H|^4-2\bigg(\frac{|A|^2}{|H|^2}-\frac{2}{n}\bigg)|\mathring{A}_H|^2|H|^2
   -\frac{2}{n}\bigg(\frac{|A|^2}{|H|^2}-\frac{1}{n}\bigg)|H|^4\\
&\ -4|\mathring{A}_H|^2|\mathring{A}_I|^2-\frac{3}{2}|\mathring{A}_I|^4\\
\leq&\ 0.
\end{split}
\end{equation}
In (\ref{3}) we have used the pinching condition
$|A|^2\leq\alpha_\epsilon|H|^2+\beta_\epsilon
c<\alpha_\epsilon|H|^2$ for $\epsilon\in(0,1)$.

By (\ref{aH2}), we have
\begin{equation}\label{4}
\begin{split}
&\
-nc|\mathring{A}|^2-\frac{an(1-\sigma)c|\mathring{A}|^2|H|^2}{a|H|^2+\beta_\epsilon
c}
\\
\leq&-nc|\mathring{A}|^2-\frac{an(1-\sigma)c|\mathring{A}|^2(a|H|^2+\beta_\epsilon
c)
}{b(a|H|^2+\beta_\epsilon c)}\\
\leq&-nc|\mathring{A}|^2-\frac{anc}{b}|\mathring{A}|^2\\
:=&\ \bar{b}|\mathring{A}|^2.
\end{split}
\end{equation}
For the last term of right hand side of (\ref{f-sigma}), we have by
(\ref{aH2})
\begin{equation}\label{5}
\begin{split}
\frac{2a\sigma R_2f_\sigma}{a|H|^2+\beta_\epsilon c}
\leq\frac{2a\sigma |H|^2|A|^2f_\sigma}{a|H|^2+\beta_\epsilon c}
\leq\frac{2a\sigma }{b}|A|^2f_\sigma:=\tilde{b}\sigma|A|^2f_\sigma.
\end{split}
\end{equation}

Combining (\ref{f-sigma}), (\ref{2}), (\ref{3}), (\ref{4}) and
(\ref{5}), we have
\begin{equation}
\begin{split}
\label{f-sigma-ineq} \frac{\partial}{\partial t}f_{\sigma}\leq&\
\triangle f_{\sigma}+\frac{2a(1-\sigma)}{a|H|^2+\beta_\epsilon
c}\langle \nabla |H|^2,\nabla
f_{\sigma}\rangle-\frac{2\epsilon_\nabla}{(a|H|^2+\beta_\epsilon
c)^{1-\sigma}}|\nabla
H|^2\\
&\ +2\bar{b}f_\sigma+\tilde{b}\sigma|A|^2f_\sigma.
\end{split}
\end{equation}

To deal with  the last term of the right hand side of
(\ref{f-sigma-ineq}), we need the following estimate.

\begin{prop}There exists a positive constant
$\varepsilon$ independent of $t$ such that
\begin{eqnarray}\label{Z-ineq}
Z+nc|\mathring{A}|^2\geq \varepsilon
|\mathring{A}|^2(a|H|^2+\beta_\epsilon c)
\end{eqnarray}
holds for $t\geq t_0$.
\end{prop}
\begin{proof}By the argument in the proof of Lemma 5.4 in
\cite{Baker}, we only have to show
\begin{eqnarray*}
-\frac{\beta_\epsilon}{\alpha_\epsilon-\frac{1}{n}}\bigg(\frac{1}{n}-\frac{n-2}{2n(n-1)}\bigg)+n\leq0.
\end{eqnarray*}
This is true by our choice of $\alpha_\epsilon$ and
$\beta_\epsilon$.
\end{proof}

\begin{prop}
For any $\eta>0$, $p\geq2$ and $t\geq t_0$, we have
\begin{equation}
\begin{split}
\label{integral-ineq} \int_{M_t}f_{\sigma}^{p}(a|H|^2+\beta_\epsilon
c)d\mu_t \leq&\
\frac{2p\eta+5}{b\varepsilon}\int_{M_t}\frac{f_{\sigma}^{p-1}}{(a|H|^2+\beta_\epsilon
c)^{1-\sigma}}|\nabla
H|^2d\mu_t\\
&\ +\frac{2(p-1)}{b\eta\varepsilon}\int_{M_t}f_\sigma^{p-2}|\nabla
f_\sigma|^2d\mu_t.
\end{split}
\end{equation}
\end{prop}

\begin{proof}
We have the following estimate.
\begin{equation}\label{1-ineq}
\begin{split}
&\ 2\int_{M_t}\frac{f_{\sigma}^{p-1}}{(a|H|^2+\beta_\epsilon
c)^{1-\sigma}}Zd\mu_t+2nc\int_{M_t}\frac{f_{\sigma}^{p-1}}{(a|H|^2+\beta_\epsilon
c)^{1-\sigma}}|\mathring{A}|^2d\mu_t\\
\leq&\
2(p-1)\int_{M_t}\frac{f_{\sigma}^{p-2}}{(a|H|^2+\beta_\epsilon
c)^{1-\sigma}}|\nabla f_\sigma||\mathring{A}||\nabla
H|d\mu_t\\
&+\frac{2(n-1)}{n}\int_{M_t}\frac{f_{\sigma}^{p-1}}{(a|H|^2+\beta_\epsilon
c)^{1-\sigma}}|\nabla H|^2d\mu_t\\
&+4\int_{M_t}\frac{f_{\sigma}^{p-1}}{(a|H|^2+\beta_\epsilon
c)^{2-\sigma}}|H||\mathring{A}||\nabla H|^2d\mu_t\\
&+4(1-\sigma)(p-2)\int_{M_t}\frac{f_{\sigma}^{p-1}}{(a|H|^2+\beta_\epsilon
c)}|H||\nabla H||\nabla f_\sigma|d\mu_t\\
&+4\int_{M_t}\frac{f_{\sigma}^{p}}{(a|H|^2+\beta_\epsilon
c)^{2}}|H|^2|\nabla H|^2d\mu_t.
\end{split}
\end{equation}
Since $|\mathring{A}|^2\leq f_\sigma(a|H|^2+\beta_\epsilon
c)^{1-\sigma}$ and $f_\sigma\leq (a|H|^2+\beta_\epsilon
c)^{\sigma}$, by choosing $\sigma\in (0,1)$ we have the following
estimates.
\begin{equation}\label{2-ineq}
\begin{split}
&\ 2(p-1)\int_{M_t}\frac{f_{\sigma}^{p-2}}{(a|H|^2+\beta_\epsilon
c)^{1-\sigma}}|\nabla f_\sigma||\mathring{A}||\nabla
H|d\mu_t\\
\leq&\ \frac{1-\sigma}{\eta}\int_{M_t}f_\sigma^{p-2}|\nabla
f_\sigma|^2d\mu_t+(p-1)\eta
\int_{M_t}\frac{f_{\sigma}^{p-1}}{(a|H|^2+\beta_\epsilon
c)^{1-\sigma}}|\nabla H|^2d\mu_t,
\end{split}
\end{equation}
\begin{equation}\label{3-ineq}
\begin{split}
&\ 4\int_{M_t}\frac{f_{\sigma}^{p-1}}{(a|H|^2+\beta_\epsilon
c)^{2-\sigma}}|H||\mathring{A}||\nabla H|^2d\mu_t\\
\leq&\
\frac{4}{b}\int_{M_t}\frac{f_{\sigma}^{p-1}}{(a|H|^2+\beta_\epsilon
c)^{1-\sigma}}|\nabla H|^2 d\mu_t,
\end{split}
\end{equation}
\begin{equation}\label{4-ineq}
\begin{split}
&\ 4(p-2)\int_{M_t}\frac{f_{\sigma}^{p-1}}{a|H|^2+\beta_\epsilon
c}|H||\nabla H||\nabla f_\sigma|d\mu_t\\
\leq&\ (p-2)\int_{M_t}\frac{1}{a|H|^2+\beta_\epsilon
c}\bigg(\frac{2}{\eta} f_{\sigma}^{p-2}|H|^2|\nabla
f_\sigma|^2+2\eta f_\sigma^p|\nabla H|^2\bigg)d\mu_t\\
\leq&\ \frac{2(p-2)}{b\eta}\int_{M_t}f_\sigma^{p-2}|\nabla
f_\sigma|^2d\mu_t\\
&\ +2(p-2)\eta\int_{M_t}\frac{f_\sigma^{p-1}}{(a|H|^2+\beta_\epsilon
c)^{1-\sigma}}|\nabla H|^2d\mu_t,
\end{split}
\end{equation}
\begin{eqnarray}\label{5-ineq}
&&4\int_{M_t}\frac{f_{\sigma}^{p}}{(a|H|^2+\beta_\epsilon
c)^{2}}|H|^2|\nabla
H|^2d\mu_t\leq\frac{4}{b}\int_{M_t}\frac{f_{\sigma}^{p-1}}{(a|H|^2+\beta_\epsilon
c)^{1-\sigma}}|\nabla H|^2d\mu_t.
\end{eqnarray}
In (\ref{3-ineq}), (\ref{4-ineq}) and (\ref{5-ineq}) we have used
(\ref{aH2}).

By (\ref{Z-ineq}), we have
\begin{equation}\label{6-ineq}
\begin{split}
&\ 2\int_{M_t}\frac{f_{\sigma}^{p-1}}{(a|H|^2+\beta_\epsilon
c)^{1-\sigma}}(Z+2nc|\mathring{A}|^2)d\mu_t\\
\geq&\ 2\varepsilon\int_{M_t}f_{\sigma}^{p}(a|H|^2+\beta_\epsilon
c)d\mu_t.\end{split}
\end{equation}

Combining (\ref{1-ineq})-(\ref{6-ineq}), we obtain
\begin{equation}\label{7-ineq}
\begin{split}
2\varepsilon\int_{M_t}f_{\sigma}^{p}(a|H|^2+\beta_\epsilon c)d\mu_t
\leq&\
\frac{3p\eta+10}{b}\int_{M_t}\frac{f_{\sigma}^{p-1}}{(a|H|^2+\beta_\epsilon
c)^{1-\sigma}}|\nabla
H|^2d\mu_t\\
&\ +\frac{3(p-1)}{b\eta}\frac{}{}\int_{M_t}f_\sigma^{p-2}|\nabla
f_\sigma|^2d\mu_t.
\end{split}
\end{equation}
Dividing through by $2\varepsilon$ completes the proof.
\end{proof}

Now we show that the $L^p$-norm of $f_\sigma$ is  bounded for
sufficiently high $p$.

\begin{lem}\label{lem1}For any $p\geq\max\{2,\frac{8}{b\epsilon_\nabla}+1\}$ and
$\sigma\leq
\min\Big\{\frac{b^2\varepsilon\epsilon_\nabla}{10\tilde{b}\alpha_\epsilon}
,\frac{b^2\varepsilon\sqrt{\epsilon_\nabla}}{4\tilde{b}\alpha_\epsilon\sqrt{p}},\frac{1}{2}\Big\}$,
there exist a constant $C$ independent of $t$ such that for all
$t\in [0,T_{\max})$ where $T_{\max}<\infty$, we have
\begin{eqnarray}\label{8-ineq}
\bigg(\int_{M_t}f_\sigma^pd\mu_t\bigg)^{\frac{1}{p}}\leq C.
\end{eqnarray}
\end{lem}

\begin{proof}For $t\geq t_0$, form (\ref{f-sigma-ineq}), we have
\begin{equation}\label{9-ineq}
\begin{split}
\frac{\partial}{\partial t}\int_{M_t}f_\sigma^pd\mu_t\leq& \
\int_{M_t}pf_\sigma^{p-1}\frac{\partial}{\partial t}f_\sigma
d\mu_t\\
\leq&\ -p(p-1)\int_{M_t}f_\sigma^{p-2}|\nabla
f_\sigma|^2d\mu_t\\
&\
+4(1-\sigma)p\int_{M_t}\frac{f_\sigma^{p-1}}{a|H|^2+\beta_\epsilon
c}|H||\nabla|H|||\nabla f_\sigma| d\mu_t\\
&\ -2p\epsilon_\nabla
\int_{M_t}\frac{f_{\sigma}^{p-1}}{(a|H|^2+\beta_\epsilon
c)^{1-\sigma}}|\nabla
H|^2d\mu_t\\
&\ +2\bar{b}p\int_{M_t}f_\sigma^{p}d\mu_t+\tilde{b}\sigma p
\int_{M_t}|A|^2f_\sigma^pd\mu_t.
\end{split}
\end{equation}
As in (\ref{4-ineq}), we have
\begin{equation}\label{10-ineq}
\begin{split}
&\ 4(1-\sigma)p\int_{M_t}\frac{f_\sigma^{p-1}}{a|H|^2+\beta_\epsilon
c}|H||\nabla|H|||\nabla f_\sigma| d\mu_t\\
\leq&\ \frac{2p}{b\mu}\int_{M_t}f_\sigma^{p-2}|\nabla
f_\sigma|^2d\mu_t+2p\mu\int_{M_t}\frac{f_\sigma^{p-1}}{(a|H|^2+\beta_\epsilon
c)^{1-\sigma}}|\nabla H|^2d\mu_t.
\end{split}
\end{equation}
Substituting (\ref{10-ineq}) to (\ref{9-ineq}),  letting
$\mu=\frac{4}{b(p-1)}$ and
$p\geq\max\{2,\frac{8}{b\epsilon_\nabla}+1\}$  we obtain
\begin{equation*}
\begin{split}
\frac{\partial}{\partial t}\int_{M_t}f_\sigma^pd\mu_t \leq&\
-\frac{p(p-1)}{2}\int_{M_t}f_\sigma^{p-2}|\nabla
f_\sigma|^2d\mu_t\\
&\ -p\epsilon_\nabla
\int_{M_t}\frac{f_{\sigma}^{p-1}}{(a|H|^2+\beta_\epsilon
c)^{1-\sigma}}|\nabla
H|^2d\mu_t\\
&\ +2\bar{b}p\int_{M_t}f_\sigma^{p}d\mu_t+\frac{\tilde{b}\sigma
\alpha_\epsilon p}{b} \int_{M_t}f_\sigma^p(a|H|^2+\beta_\epsilon
c)d\mu_t.
\end{split}
\end{equation*}
This together with (\ref{integral-ineq}) implies
\begin{equation}\label{11-ineq}
\begin{split}
\frac{\partial}{\partial t}\int_{M_t}f_\sigma^pd\mu_t \leq&\
-p(p-1)\bigg(\frac{1}{2}-\frac{2\tilde{b}\sigma
\alpha_\epsilon}{b^2\eta\varepsilon
}\bigg)\int_{M_t}f_\sigma^{p-2}|\nabla
f_\sigma|^2d\mu_t\\
&\ -\bigg(p\epsilon_\nabla-\frac{(2p\eta+5){\tilde{b}\sigma
\alpha_\epsilon p}}{b^2\varepsilon}\bigg)
\int_{M_t}\frac{f_{\sigma}^{p-1}}{(a|H|^2+\beta_\epsilon
c)^{1-\sigma}}|\nabla
H|^2d\mu_t\\
&\ +2\bar{b}p\int_{M_t}f_\sigma^{p}d\mu_t.
\end{split}
\end{equation}
Now we pick
$\eta=\frac{4\tilde{b}\alpha_\epsilon\sigma}{b^2\varepsilon}$ and
let $\sigma\leq
\min\Big\{\frac{b^2\varepsilon\epsilon_\nabla}{10\tilde{b}\alpha_\epsilon}
,\frac{b^2\varepsilon\sqrt{\epsilon_\nabla}}{4\tilde{b}\alpha_\epsilon\sqrt{p}},\frac{1}{2}\Big\}$.
Then (\ref{11-ineq}) reduces to
\begin{eqnarray*}
\frac{\partial}{\partial t}\int_{M_t}f_\sigma^pd\mu_t
\leq2\bar{b}p\int_{M_t}f_\sigma^{p}d\mu_t.
\end{eqnarray*}
This implies
\begin{eqnarray}
\int_{M_t}f_\sigma^pd\mu_t\leq
e^{2\bar{b}pt}\int_{M_{t_0}}f_\sigma^{p}d\mu_t.
\end{eqnarray}

If $t\in [0,t_0]$, by the smoothness of the mean curvature flow we
see that $\int_{M_t}f_\sigma^pd\mu_t$ is bounded. For $t\geq t_0$,
we only have to show that $T_{\max}$ is finite.

\begin{lem}\label{finite-time}The maximal existence time $T_{\max}$ of the mean curvature flow is
finite.
\end{lem}
\begin{proof}Fixed a point $y\in \mathbb{F}^{n+d}(c)$ and let
$r$ be the distance function on $\mathbb{F}^{n+d}(c)$ from $y$.
Denote also by $r$ the composition $r\circ F_t$. We may assume that
$r>0$ on $M_t$ for $t\in [0,T_{\max})$. In fact, if $d=1$, we may
choose $y$ such that it is outside of a geodesic ball in
$\mathbb{F}^{n+1}(c)$ that encloses $M_0$. By the maximum principle
we see that $y$ doesn't lies in any $M_t$. If $d>1$, then the
Haussdorff dimension of $F(M\times [0,T_{\max}))$ is no more than
$n+1$. So we can also pick a point $y$ such that it doesn't lies in
any $M_t$. In both cases, we have $r>0$ on each  $M_t$.

From \cite{CGM}, we know that
\begin{eqnarray}\label{12}\triangle r=\langle
H,\partial_r\rangle+{\rm co_c}(r)(n-|\partial_r^T|^2).
\end{eqnarray}
Here ${\rm co_c}(r)=\frac{\sqrt{-c}\cosh
(\sqrt{-c}r)}{\sinh(\sqrt{-c}r)}=\frac{\sqrt{-c}(e^{\sqrt{-c}r}+e^{-\sqrt{-c}r})}{e^{\sqrt{-c}r}-e^{-\sqrt{-c}r}}$,
$\partial_r$ is the gradient of $r$ in $\mathbb{F}^{n+d}(c)$, and
$\partial_r^T$ is the tangent part of $\partial_r$ to $M_t$. Clearly
we have ${\rm co_c}(r)\geq \sqrt{-c}$ and $|\partial_r^T|^2\leq1$.

On the other hand, since $F_t$ satisfies (\ref{MCF}), we have
\begin{eqnarray}\label{13-0}\frac{\partial}{\partial t} r=\langle H,
\partial_r\rangle.
\end{eqnarray}
Combining (\ref{12}) and (\ref{13-0}) we obtain
\begin{eqnarray}\label{14}\frac{\partial}{\partial t} r= \triangle
r-{\rm co}_c(r)(n-|\partial_r^T|^2).
\end{eqnarray}
Suppose that $r(0)<R$. By the maximum principle  we see that
\begin{eqnarray}\label{15}r(t)< R-(n-1)\sqrt{-c}t.
\end{eqnarray}
Then $T_{\max} < \frac{R}{(n-1)\sqrt{-c}}$, i.e., the maximal
existence time of the mean curvature flow is finite.
\end{proof}

By Proposition \ref{finite-time}, we finish the proof of Lemma
\ref{lem1}.

\end{proof}

Now we can proceed as in \cite{H1} or \cite{HS} via a Stampacchia
iteration procedure to complete the proof of Theorem \ref{thm1}.

\end{proof}

\section{A gradient estimate for the mean curvature}

We establish a gradient estimate for the mean curvature flow, which
will be used to compare the mean curvature at different points of
the submanifold. We also assume that $c<0$.

\begin{thrm}\label{gradient-H}For every $\eta>0$, there exists a
constant $C_\eta$ independent of $t$ such that for all $t\in
[0,T_{\max})$, there holds
\begin{eqnarray}\label{16}|\nabla H|^2\leq \eta |H|^4+C_\eta.
\end{eqnarray}
\end{thrm}
\begin{proof} By direct computation, we have
\begin{eqnarray}\label{19}
\frac{\partial}{\partial t}|H|^4\geq\triangle |H|^4-12|H|^2|\nabla
H|^2+\frac{4}{n}|H|^6+4nc|H|^4,
\end{eqnarray}
\begin{eqnarray}\label{17}
\frac{\partial}{\partial t}|\nabla H|^2\leq \triangle |\nabla
H|^2+C_1|H|^2|\nabla A|^2+C_2|\nabla A|^2,
\end{eqnarray}
for constants $C_1$ and $C_2$ independent of $t$.

We also have the following estimate for sufficiently large positive
constants $N_1$ and $N_2$ independent of $t$.
\begin{equation}\label{18}
\begin{split}
\frac{\partial}{\partial t}\bigg((N_1+N_2|
H|^2)|\mathring{A}|^2\bigg)\leq&\ \triangle\bigg((N_1+N_2|
H|^2)|\mathring{A}|^2\bigg)-\frac{4(n-1)}{3n}(N_2-1)|H|^2|\nabla
A|^2\\
&\ -\frac{4(n-1)}{3n}(N_1-C(N_2))|\nabla
A|^2\\
&\ -C_2(N_1,N_2)|\mathring{A}|^2(|H|^4+1)-2ncN_1|\mathring{A}|^2.
\end{split}
\end{equation}
In (\ref{18}), $C(N_2)$ and $C(N_1, N_2)$ are constants depending on
$N_2$ and $N_1,\ N_2$ respectively. Consider the function $f=|\nabla
H|^2+(N_1+N_2| H|^2)|\mathring{A}|^2-\eta|H|^4$. From (\ref{19}),
(\ref{17}) and (\ref{18}), we have
\begin{equation}\label{20}
\begin{split}
\frac{\partial}{\partial t}f\leq&\ \triangle
f-\frac{4(n-1)}{3n}(N_2-1)|H|^2|\nabla A|^2-\frac{4(n-1)}{3n}
(N_1-C_3(N_2))|\nabla A|^2\\
&\ +C_4(N_1, N_2) |\mathring{A}|^2(|H|^4+1)+12\eta|H|^2|\nabla
H|^2-\frac{4\eta}{n}|H|^6-4nc\eta |H|^4.
\end{split}
\end{equation}
Here we have consumed  $C_1|H|^2|\nabla A|^2+C_2|\nabla A|^2$ by
firstly choosing sufficiently large $N_2$ and secondly choosing
sufficiently large $N_1$. Notice that $|\nabla H|^2\leq n|\nabla
A|^2$. We can choose larger $N_2$ and $N_1$ depending on $\eta$ to
consume $12\eta|H|^2|\nabla H|^2$ and make the second, third terms
of the right hand side of (\ref{20}) negative. Since
$|\mathring{A}|^2\leq C_0|H|^{2-\sigma_0}$ for $t\geq t_0$ and
$|\mathring{A}|^2$ is uniformly bounded for $t\in [0,t_0]$, using
Young's inequality we get
\begin{eqnarray*}
C_4(N_1, N_2) |\mathring{A}|^2(|H|^4+1)-4nc\eta |H|^4\leq
\frac{4\eta}{n}|H|^6+C_\eta.
\end{eqnarray*}
Here $C_\eta$ is a constant depending on $\eta$ and other quantities
but independent of $t$. Then we obtain
\begin{eqnarray*}
\frac{\partial}{\partial t}f\leq \triangle f+C_5.
\end{eqnarray*}
Notice that $T_{\max}$ is finite. Then the theorem follows from the
maximum principle and the definition of $f$.
\end{proof}

\section{Convergence of MCF  in a hyperbolic space}

\begin{thrm}\label{convergence-c<0}Let $F:M^{n}\rightarrow
\mathbb{F}^{n+d}(c)$ be a smooth closed submanifold, where $n\geq 2$
and $c<0$. Assume $F$ satisfies
\begin{eqnarray}
\label{pinch-cond-2} |A|^2\leq\begin{cases}
            \frac{4}{3n}|H|^2+\frac{n}{2}c, \ &n = 2, 3, \\
            \frac{1}{n-1}|H|^2+2c, \ &n \geq 4.
        \end{cases}
\end{eqnarray}
Then $F_t(M)$ converges to a round point in finite time.
\end{thrm}
\begin{proof}
By the curvature estimate in \cite{Chen}, we see that
\begin{eqnarray*}
K_{\min}(x)\geq\frac{1}{2}\bigg(\frac{1}{n-1}-\alpha_\epsilon\bigg)|H|^2(x)+\frac{1}{2}(2-\beta_\epsilon)c.
\end{eqnarray*}
By our choices of $\alpha_\epsilon$ and $\beta_\epsilon$, and the
preserved pinching condition, we see that there exists a positive
constant $\varepsilon_0$  independent of $t$ such that
\begin{eqnarray}\label{21}
K_{\min}(x)\geq\varepsilon_0|H|^2.
\end{eqnarray}

Since $T_{\max}$ is finite, $\max_{M_t}|A|^2\rightarrow \infty$ as
$t\rightarrow T_{\max}$. By similar arguments as in
\cite{Andrews-Baker,H1,H2} we have
$\frac{\max_{M_t}|H|}{\min_{M_t}|H|}\rightarrow 1$ as $t\rightarrow
T_{\max}$, and $M_t$'s converge to a single point $o$ as
$t\rightarrow T_{\max}$. If we take a rescaling around $o$ (since
$\mathbb{F}^{n+d}(c)$ can be consider as a linear space that
isomorphic to $\mathbb{R}^{n+d}$) such that the total area of the
expanded submanifolds are fixed, then the rescaled immersions
converge to a totally umbilical immersion as $t\rightarrow
T_{\max}$.
\end{proof}

When $p=1$ and $n=3$, we have the following proposition.

\begin{prop}Let $F:M^{3}\rightarrow \mathbb{F}^{4}(c)$ be a smooth closed
hypersurface in a hyperbolic space with constant curvature $c<0$.
Assume $F$ satisfies
\begin{eqnarray*}
|A|^2\leq\frac{1}{2}|H|^2+2c.
\end{eqnarray*}
Then the mean curvature flow with $F$ as initial value converges to
a round point in finite time.
\end{prop}

\begin{proof}
 If $d=1$, then $\mathring{A}_I=0$. If we take $\alpha\leq
\frac{1}{n-1}$ and $\beta\geq 2$ for all $n\geq 2$, the left hand
side of (\ref{1}) is nonpositive. Hence $|A|^2\leq
\frac{1}{n-1}|H|^2+2c$ is preserved along the mean curvature flow
for all $n\geq 2$. When $n=3$, if we set
$a=\frac{1}{3(2+\epsilon)}$, then
$\epsilon_\nabla=\frac{3}{4}-\frac{1}{3} -a>0$, and Theorem
\ref{thm1} also holds. Then Theorem \ref{gradient-H} follows. By a
similar argument as in the proof of Theorem \ref{convergence-c<0},
we get the convergence of the mean curvature flow.

\end{proof}

\begin{remark} When $d=1$ and $n=3$, the pinching condition
$|A|^2\leq \frac{1}{2}|H|^2+2c$ is better than condition
(\ref{pinch-cond-2}). In fact, we have
$\frac{1}{2}|H|^2+2c-\frac{4}{9}|H|^2-\frac{3}{2}c=\frac{1}{18}|H|^2+\frac{1}{2}c>0$.
Here we have used the fact that $|H|^2+9c>0$, which is implied by
$|A|^2\leq \frac{1}{2}|H|^2+2c$.
\end{remark}

\end{document}